\documentclass[11pt, leqno]{amsart}
\usepackage{mathrsfs}
\usepackage{graphicx}
\usepackage{amsfonts,delarray,amssymb,amsmath,amsthm, a4,a4wide}
\usepackage{latexsym}
\usepackage{epsfig}
\usepackage{color}
\usepackage[T1]{fontenc}

\newtheorem{thm}{Theorem}[section]

\newtheorem{lem}[thm]{Lemma}
\newtheorem{prop}[thm]{Proposition}

\theoremstyle{remark}
\newtheorem{rem}[thm]{Remark}

\theoremstyle{definition}
\newtheorem{defn}[thm]{Definition}

\numberwithin{equation}{section}

\newcommand{\R}{\mathbb R}
\newcommand{\K}{\mathbb K}

\newcommand{\p}{\partial}

\newcommand{\comment}[1]{}
\def\h{\hspace*{.24in}}

\makeatletter
\@namedef{subjclassname@2020}{%
  \textup{2020} Mathematics Subject Classification}
\makeatother
\begin{document} 

\title[Equality case in Brunn--Minkowski inequality for Monge--Amp\`ere eigenvalue ]{On the equality case in the Brunn--Minkowski inequality for the Monge--Amp\`ere eigenvalue 
}
\author{Nam Q. Le}
\address{Department of Mathematics, Indiana University, 
Bloomington, IN 47405, USA}
\email{nqle@iu.edu}
\thanks{The author was supported in part by the National Science Foundation under grant DMS-2452320.}



\begin{abstract}
We characterize the equality case in the Brunn--Minkowski inequality for the Monge--Amp\`ere eigenvalue of general bounded convex domains. This together with the author's previous work completely resolves a question raised by Salani (A Brunn--Minkowski inequality for the Monge--Amp\`ere eigenvalue. {\it Adv. Math.} {\bf 194} (2005)). 
\end{abstract}
\maketitle
\section{The Brunn--Minkowski inequality for the Monge--Amp\`ere eigenvalue}

The eigenvalue problem for the Monge--Amp\`ere operator on smooth, uniformly convex domains $\Omega$ in $\R^n$ ($n\geq 2$) was 
first investigated by Lions \cite{Ls}. He showed that there exist a unique positive constant $\lambda(\Omega)$ and a unique (up to positive multiplicative constants) nonzero 
convex function $u\in C^{1,1}(\overline{\Omega})\cap C^{\infty}(\Omega)$ solving the Monge--Amp\`ere eigenvalue problem: 
\begin{equation}\label{EP}
\left\{
 \begin{alignedat}{2}
   \det D^2 u~& = \lambda(\Omega)|u|^n~&&\text{in} ~  \Omega, \\\
u &= 0~&&\text{on}~ \p\Omega.
 \end{alignedat} 
  \right.
  \end{equation}
 The constant $\lambda(\Omega)$ is called 
the Monge--Amp\`ere eigenvalue of $\Omega$ and the nonzero convex functions $u$ solving (\ref{EP}) are called the Monge--Amp\`ere eigenfunctions of $\Omega$.

\medskip
For a convex function $u\in C(\overline{\Omega})$ on a bounded convex domain $\Omega\subset\R^n$, 
we define its Rayleigh quotient by
\begin{equation}
\label{RQ}
R(u) = \frac{\int_{\Omega} (-u)\, d\mu_u}{\int_{\Omega} (-u)^{n+1}\,dx}.
\end{equation}
Here $\mu_u$ denotes the Monge--Amp\`ere measure associated with $u$; see Definition \ref{MAdef}. When $u\in C^2(\Omega)$, we have $d\mu_u = \det D^2 u\, dx$.

\medskip
A variational characterization of $\lambda(\Omega)$ was found by Tso \cite{Tso} who discovered that for sufficiently smooth, uniformly convex domains $\Omega\subset\R^n$,
\begin{equation}
 \label{lamT}
 \lambda(\Omega)=\inf\Big\{ R(u): u\in C^{0,1}(\overline{\Omega})\cap C^2(\Omega),~u~\text{is nonzero, convex in } 
 \Omega,~u=0~\text{on}~\p\Omega\Big\}.
\end{equation}

\medskip
As a function of smooth, uniformly convex domains in $\R^n$, $\lambda(\cdot)$ is positively homogeneous of degree $(-2n)$ in the sense that \[\lambda (t\Omega)= t^{-2n}\lambda (\Omega) \quad\text{for all } t>0.\] 
Using (\ref{lamT}), 
Salani \cite{Sal} proved a Brunn--Minkowski inequality
for $\lambda(\cdot)$ by showing that $\lambda^{-\frac{1}{2n}}(\cdot)$ (which is positively homogeneous of degree $1$) is concave in the class of 
$C^2$, uniformly convex domains in $\R^n$ endowed with the Minkowski addition. More precisely, Salani proved that if $\Omega_0$ and $\Omega_1$ are uniformly convex domains in $\R^n$  with $C^{2}$  boundaries, 
then
\begin{equation}
 \lambda(\Omega_{\alpha})^{-\frac{1}{2n}}\geq (1-\alpha) \lambda (\Omega_0)^{-\frac{1}{2n}} + \alpha \lambda(\Omega_1)^{-\frac{1}{2n}}\quad\quad\text{for all}~\alpha\in (0, 1),
 \label{lamBM}
\end{equation}
where $\Omega_{\alpha}$ denotes the Minkowski linear combination of $\Omega_0$ and $\Omega_1$ with weight $\alpha$:
\begin{equation}
\label{Min_lin}
 \Omega_{\alpha}= (1-\alpha)\Omega_0 + \alpha \Omega_1=\big\{(1-\alpha)x_0 + \alpha x_1: x_0\in\Omega_0, x_1\in\Omega_1\big\}.
\end{equation}
Moreover, equality occurs in \eqref{lamBM} if and only if $\Omega_1$ is homothetic to $\Omega_0$, that is, there exist $\beta>0$ and $b\in\R^n$ such that $\Omega_1= \beta\Omega_0 + b$. 

\medskip For the classical Brunn--Minkowski inequality for volume of convex bodies, see Schneider \cite[Chapter 7]{Sc} for an exposition and extension.

\medskip
Salani \cite[p. 85]{Sal} asked if the  Brunn--Minkowski inequality (\ref{lamBM}) and its equality case also hold for general bounded convex domains. 
In order to answer this question, we need to study the Monge--Amp\`ere eigenvalue for general  bounded convex domains. This was  done in  \cite{L}.

\medskip
\noindent
Let
\[
\K = \{  w \in C(\overline{\Omega}):  
 ~w~\text{is convex, nonzero in } \Omega,~ w=0~\text{on}~\p\Omega \}.
\]

\medskip
For general bounded convex domains $\Omega\subset\R^n$, the existence, uniqueness and variational characterization of the Monge--Amp\`ere eigenvalue, and uniqueness of  convex Monge--Amp\`ere eigenfunctions 
were obtained in \cite{L}. 
 They are
the singular counterparts of those established by  Lions \cite{Ls} and Tso \cite{Tso} in the smooth, uniformly convex setting. 
For the purpose of this note, we recall  here  \cite[Theorem 1.1]{L} and \cite[Theorem 1.1]{L_AFST}.
\begin{thm}[The Monge--Amp\`ere eigenvalue problem for general bounded convex domains]  
\label{ev_thm}
 Let $\Omega$ be a bounded convex domain in $\R^n$. Define 
\begin{equation}
\label{lam_def}
 \lambda[\Omega] =\inf_{w\in \K} \frac{\int_{\Omega} (-w)\, d\mu_w}{\int_{\Omega} (-w)^{n+1}\,dx}.
 \end{equation}
 Then, the following facts hold.
 \begin{enumerate}
 \item[(i)] (Existence) The infimum in (\ref{lam_def}) is achieved by a 
 nonzero convex solution $w\in C^{0, 1}(\overline{\Omega})\cap C^{\infty}(\Omega)$ to the eigenvalue problem 
  \begin{equation}
   \label{EVP_eq}
 \left\{
 \begin{alignedat}{2}
   \det D^{2} w~&=\lambda[\Omega] |w|^{n} \h~&&\text{in} ~\Omega, \\\
w &=0\h~&&\text{on}~\p \Omega.
 \end{alignedat}
 \right.
\end{equation}
 The constant $\lambda[\Omega]$ is called the Monge--Amp\`ere eigenvalue of $\Omega$ and $w$ is called a Monge--Amp\`ere eigenfunction of $\Omega$.
\item[(ii)] (Uniqueness) If the pair $(\Lambda, \tilde w)$ 
satisfies $\det D^2 \tilde w =\Lambda |\tilde w|^n$ in $\Omega$ where $\Lambda$ is a positive constant and 
$\tilde w\in \K$, then $\Lambda=\lambda[\Omega]$ and $\tilde w=m w$ for some positive constant $m$.
\item[(iii)] (Stability) $\lambda[\cdot]$ is stable with respect to the Hausdorff convergence of the domains: If $\{\Omega_m\}_{m=1}^{\infty}\subset\R^n$ is a sequence of bounded convex domains that 
converges in the Hausdorff distance to a bounded convex domain $\Omega$, then $\lim_{m\rightarrow \infty}\lambda[\Omega_m]=\lambda[\Omega].$
\end{enumerate}
\end{thm}

When $\Omega$ is uniformly convex with smooth boundary, by uniqueness, $\lambda(\Omega)$ defined by \eqref{lamT} and  $\lambda[\Omega]$ defined by \eqref{lam_def} must be equal, though the set of competitors for $\lambda(\Omega)$ is strictly contained in the set of competitors $\K$ for $\lambda[\Omega]$.  
We chose the bracket notation $\lambda[\Omega]$ for the Monge--Amp\`ere eigenvalue of a general bounded convex domain $\Omega$ to emphasize that its boundary might have flat parts or corners.

\medskip
Using the stability with respect to the Hausdorff convergence of the Monge--Amp\`ere eigenvalue and the fact that any bounded convex domain can be approximated in the Hausdorff distance by a sequence of smooth, uniformly convex domains, we can immediately extend Salani's Brunn--Minkowski inequality (\ref{lamBM})
to bounded convex domains; see \cite[Theorem 1.3]{L}.
\begin{thm}[The Brunn--Minkowski inequality for the Monge--Amp\`ere eigenvalue]
\label{BMcor}
Let $\Omega_0$ and $\Omega_1$ be bounded convex domains in $\R^n$ and let $\alpha\in (0, 1)$. Then 
\[\lambda\big[(1-\alpha)\Omega_0 + \alpha\Omega_1\big]^{-\frac{1}{2n}}\geq (1-\alpha) \lambda [\Omega_0]^{-\frac{1}{2n}} + \alpha \lambda[\Omega_1]^{-\frac{1}{2n}}.\]
\end{thm}
Theorem \ref{BMcor} is the Monge--Amp\`ere analogue of the Brunn--Minkowski inequality for the eigenvalue of the Laplace operator on bounded domains established by Brascamp--Lieb \cite[Theorem 6.2]{BL} (see also Borell \cite[Theorem 3.2]{Bo} for a different proof). 

\medskip
The case of equality in Theorem \ref{BMcor} was left open in \cite{L}. 
Here, we resolve this issue.
\begin{thm} [Equality case of the Brunn--Minkowski inequality ]
\label{eqthm}
Equality holds in Theorem \ref{BMcor} for $\alpha\in (0, 1)$ if and only if $\Omega_1$ is homothetic to $\Omega_0$. 
\end{thm}
Theorems \ref{BMcor} and \ref{eqthm} completely resolve the question raised by Salani in  \cite[p. 85]{Sal}.

\medskip
\noindent
To prove Theorem \ref{eqthm}, we first directly prove Theorem \ref{BMcor} without using approximations.
Our proof follows the general strategy in Salani \cite{Sal} using the Pr\'ekopa--Leindler inequality and the infimum convolution of convex functions to create a suitable competitor for the Monge--Amp\`ere eigenvalue
$\lambda\big[(1-\alpha)\Omega_0 + \alpha\Omega_1\big]$ from the eigenfunctions of $\Omega_0$ and $\Omega_1$. These domains are no longer assumed to be strictly convex so some assertions in \cite{Sal} might not hold anymore. We need to find appropriate replacements. The key examples are:
\begin{enumerate}
\item Proposition \ref{obs1} replaces Proposition 9 in \cite{Sal} where the inclusion was in fact an equality. 
\item The sublinearity of the Monge--Amp\`ere energy $\int_\Omega |u| d\mu_u$ under the infimum convolution of convex functions in Proposition \ref{obs4} replaces Proposition 14 in \cite{Sal} where the linearity was established. 
\end{enumerate}
Once Theorem \ref{BMcor}  is proved with the above replacements, the equality case can be deduced from the equality case of the Pr\'ekopa--Leindler inequality. Interestingly and differently from \cite{Sal}, our proof does not use any property on the facial structure of convex bodies.

\medskip
The rest of this note is devoted to presenting self-contained proofs of Theorems \ref{BMcor} and \ref{eqthm}. 
Section \ref{sec_inf} will discuss the infimum convolution of convex functions and its effect on the Monge--Amp\`ere energy.  The proofs of Theorems \ref{BMcor} and \ref{eqthm} will be given in Section \ref{sec_pf}.

\section{Monge--Amp\`ere measure and infimum convolution of convex functions}
\label{sec_inf}
In this section, we study the infimum convolution of convex functions and its effect on the Monge--Amp\`ere energy.
\subsection{Monge--Amp\`ere measure and Legendre transform}
For a convex function \(u:\Omega \to \R\) on an open set $\Omega\subset\R^n$, we define
 the subdifferential of $u$  at $x\in\Omega$ by
 \[
\partial u (x):=\Big\{p\in \R^{n}\,:\, u(y)\ge u(x)+p\cdot (y-x)\quad \text{for all } y \in \Omega\Big\}.
\]
Then $\p u(x)\neq \emptyset$ for each $x\in\Omega$. For a Borel set $E\subset\Omega$, we define
\[\p u(E)=\bigcup_{x\in E} \p u(x).\]
Note that $u$ is globally Lipschitz if and only if its total subdifferential  $\p u(\Omega)$ is bounded.

We use $|E|$ to denote the Lebesgue measure of a measurable set $E\subset\R^n$. Below is a precise definition of the Monge--Amp\`ere measure of a convex function  \(u:\Omega \to \R\); see \cite[Definition 3.1]{Lbook}.
\begin{defn}[Monge--Amp\`ere measure]
\label{MAdef}
Let $u:\Omega\rightarrow \R$ be a convex function on an open set $\Omega\subset\R^n$. The Monge--Amp\`ere measure, $\mu_u$, associated with $u$ is defined by
\[\mu_u(E) = |\p u(E)|\quad
\text{for each Borel set } E\subset\Omega.\]
If $u\in C^{2}(\Omega)$, then 
$
\mu_u=\det D^{2} u(x)\,dx$ in $\Omega$.
\end{defn}
The numerator in \eqref{lam_def} is the Monge--Amp\`ere energy \[\int_\Omega (-w) \,d\mu_w\] of the convex function $w\in \K$.

\medskip
The following notion will be very useful in our analysis.
\begin{defn}[Legendre transform]
For a convex function $u\in C(\overline{\Omega})$ on a bounded convex domain $\Omega$ in $\R^n$, let $u^\ast:\p u(\Omega)\to\R$ be its Legendre transform:
\[u^\ast(p)=\max\Big\{p\cdot x- u(x): x\in \overline{\Omega}\Big\}, \quad p\in \p u(\Omega).\]
\end{defn}
As a maximum of a family of linear functions, $u^*$ is convex.  If $p\in \p u(x)$ where $x\in\Omega$, then $x\in \p u^*(p)$ and vice versa.  
Since $u^*$ is locally Lipschitz in $\p u(\Omega)$, by the Rademacher theorem, the gradient $Du^*(p)$ exists for $p\in \p u(\Omega)$  almost everywhere in the sense of Lebesgue measure.  If $u$ is strictly convex, then $\p u^*(p)$ is a singleton for each $p\in \p u(\Omega)$ so $u^*$ is differentiable everywhere in $\p u(\Omega)$; consequently, $u^*\in C^1(\p u(\Omega))$ and $\p u^*(p) =\{Du^*(p)\}$ for each $p\in \p u(\Omega)$ (see \cite[Theorem 2.38]{Lbook}).

\medskip
We can verify that for a (not necessary strictly) convex function $u$, the Monge--Amp\`ere measure $\mu_u$ is the push-forward by the gradient  map $Du^*$ of the  Lebesgue measure $dx$ (see also Villani \cite[Lemma 4.6]{V}).  This means that 
\begin{equation}
\label{MALeg}
|(Du^*)^{-1}(E)|=|\p u(E)|\quad\text{for each Borel set } E\subset\Omega.\end{equation}

\medskip
Indeed, if $p\in (Du^*)^{-1}(x)$ where $x\in E$, then $x= Du^*(p)\subset \p u^* (p)$, so $p\in \p u(x) \subset \p u(E)$. This shows that
\begin{equation} \label{MALeg1} (Du^*)^{-1}(E)\subset \p u(E).\end{equation}
From this argument, we see that if $p\in \p u(E)
\setminus (Du^*)^{-1}(E)$, then $p$ is a point of nondifferentiability of $u^*$. The set of such points has Lebesgue measure zero. Hence
\begin{equation} \label{MALeg2}|\p u(E)
\setminus (Du^*)^{-1}(E)|=0.
\end{equation}
From \eqref{MALeg1} and \eqref{MALeg2}, we obtain \eqref{MALeg} as asserted.

\medskip
From  \eqref{MALeg}, we have the following change of variables formula.
\begin{lem} 
\label{change_lem}
Let $u: \Omega\to\R$ be a convex function on a bounded convex domain $\Omega$ in $\R^n$, and let $u^\ast:\p u(\Omega)\to\R$ be its Legendre transform. Then
\begin{equation}
\label{change_var}
\int_{\Omega} f\, d\mu_u= \int_{\p u(\Omega)} f(D u^\ast (p))\, dp\quad \text{for all } f\in C(\Omega).\end{equation}
\end{lem}
Note that \eqref{change_var} coincides with \cite[Equation (19)]{Sal} for the case of $u$ being strictly convex.

\medskip
\noindent
Throughout the rest of this section, we assume 
   \begin{equation}
\label{Sta}
  \left\{\begin{alignedat}{1}
  &\mbox{$\Omega_0$ and $\Omega_1$ are bounded convex domains in $\R^n$},\\
 &  \Omega_\alpha =(1-\alpha)\Omega_0 +\alpha \Omega_1 \quad \text{where } \alpha\in (0, 1),\\
&  \mbox{$u_i\in C(\overline{\Omega_i})$ is a strictly convex function 
in $\Omega_i$ that vanishes on $\p\Omega_i$ for $i=0, 1$}. &&
  \end{alignedat}\right.
\end{equation}

\subsection{Infimum convolution of two convex functions and Monge--Amp\`ere energy}
A key notion in the proof of Theorem \ref{BMcor} is the {\it infimum convolution of two convex functions}.
\begin{defn}[Infimum convolution of two convex functions] Assume \eqref{Sta}. We define the infimum convolution $u_\alpha$ on $\Omega_\alpha$ of $u_0$ and $u_1$ by
\begin{equation}
\label{infc}
u_\alpha(x) =\min\Big\{ (1-\alpha)u_0(x_0) +\alpha u_1(x_1):  x_0\in  \overline{\Omega_0}, x_1\in  \overline{\Omega_1}, x= (1-\alpha)x_0 +\alpha x_1\Big\}. \end{equation}
\end{defn} 
When $u_0$ and $u_1$ are  Monge--Amp\`ere eigenfunctions of $\Omega_0$ and $\Omega_1$, respectively, we will use the infimum convolution $u_\alpha$ as a competitor in the variational formula \eqref{lam_def}
for the Monge--Amp\`ere eigenvalue $\lambda[\Omega_\alpha]$. Therefore, we need to understand its fine properties. 

\medskip
We recall \cite[Lemma 8]{Sal} which uses only the strict convexity of the functions $u_0$ and $u_1$ but not of the domains.
\begin{lem} 
\label{Sal8}
Assume \eqref{Sta}. Let $u_\alpha$ be as in \eqref{infc}. Then, the following statements hold.
\begin{enumerate}
\item For every $x\in \overline{\Omega_\alpha}$, there exists a unique couple of points $(x_0, x_1)\in \overline{\Omega_0} \times \overline{\Omega_1}$ such that
\[x= (1-\alpha)x_0 +\alpha x_1\quad \text{and}\quad u_\alpha(x)=(1-\alpha)u_0(x_0) +\alpha u_1(x_1).\]
\item $u_\alpha \in  C(\overline{\Omega_\alpha})$ is strictly convex with the following property:
\begin{equation}
\label{uastrict}
u_\alpha = 0\quad \text{on }\p\Omega_\alpha,\quad \text{and } u_\alpha < 0\quad \text{in }\Omega_\alpha.\end{equation}
\end{enumerate}
\end{lem}
\begin{proof} We sketch the proof for the reader's convenience. 

Fix $x\in \overline{\Omega_\alpha}$. Since $u_0$ and $u_1$ are strictly convex, the function $(x_0, x_1)\mapsto (1-\alpha)u_0(x_0) +\alpha u_1(x_1)$ is strictly convex in the convex 
set $\big\{(x_0, x_1) \in  \overline{\Omega_0} \times \overline{\Omega_1}: x= (1-\alpha)x_0 +\alpha x_1\big\}$. 
Thus, the minimum in \eqref{infc} is uniquely attained. This prove part (i).

\medskip
We now prove part (ii). Clearly, $u_\alpha \in  C(\overline{\Omega_\alpha})$. The strict convexity of $u_\alpha$ follows from that of $u_0$ and $u_1$ and the representation in part (i). It remains to verify that 
$u_\alpha = 0$ on $\p\Omega_\alpha$, but this follows from the definition and the fact that $\p\Omega_\alpha\subset (1-\alpha)\p\Omega_0 +\alpha \p\Omega_1$.
\end{proof}
Our main observation is that {\it the total subdifferential of the infimum convolution of two convex functions is contained in the union of their total subdifferentials while it contains their intersection}. 
Propositions \ref{obs1} and \ref{obs2} will present this result together with additional properties. 

 Before stating a more descriptive version of this result, we begin with some preliminaries. Assume $p\in \p u_{\alpha}(\Omega_\alpha)$. Then, by the strict convexity of $u_\alpha$, $p\in \p u_{\alpha}(x_\alpha)$ for a unique $x_\alpha\in\Omega_\alpha$. By Lemma \ref{Sal8}, 
there exists a unique couple of points $(x_0, x_1)\in \overline{\Omega_0} \times \overline{\Omega_1}$ such that
\[x_\alpha= (1-\alpha)x_0 +\alpha x_1\quad \text{and}\quad u_\alpha(x_\alpha)=(1-\alpha)u_0(x_0) +\alpha u_1(x_1).\]
Due to \eqref{uastrict}, we have $u_\alpha(x_\alpha)<0$ so $(x_0, x_1)\not\in \p\Omega_0\times \p\Omega_1$.  
\begin{prop} 
\label{obs1}
Assume \eqref{Sta}. Let $u_\alpha$ be as in \eqref{infc}. Let $p\in \p u_{\alpha}(x_\alpha)$ where $x_\alpha\in\Omega_\alpha$ with a unique couple of points $(x_0, x_1)\in \overline{\Omega_0} \times \overline{\Omega_1}\setminus  (\p\Omega_0\times \p\Omega_1)$ such that
\[x_\alpha= (1-\alpha)x_0 +\alpha x_1\quad \text{and}\quad u_\alpha(x_\alpha)=(1-\alpha)u_0(x_0) +\alpha u_1(x_1).\]
\begin{enumerate}
\item  If $x_0\in\Omega_0$, then 
\[p\in \p u_0(x_0)\quad\text{and }x_0= Du_0^*(p).\]
 If $x_1\in\Omega_1$, then 
\[p\in \p u_1(x_1)\quad\text{and }x_1= Du_1^*(p).\]
Consequently,
\[\p u_{\alpha}(\Omega_\alpha) \subset \p u_{0}(\Omega_0)\cup  \p u_{1}(\Omega_1).\]
\item If $x_1\in\p\Omega_1$ then $p\not \in \p u_{1}(\Omega_1)$. 
Similarly, if $x_0\in\p\Omega_0$ then $p\not \in \p u_{0}(\Omega_0)$. 
\end{enumerate}
\end{prop}
\begin{proof} 

We prove part (i). Assume $x_0\in\Omega_0$. We show that $p\in \p u_0(x_0)\subset  \p u_{0}(\Omega_0)$.
Indeed, let $z_0\in\Omega_0$. Consider $z_\alpha: = (1-\alpha)z_0 +\alpha x_1\in \Omega_\alpha$.
From the definitions of $u_\alpha$ and $p$, we have
\begin{equation*}
\begin{split}
(1-\alpha)u_0(z_0) +\alpha u_1(x_1)&\geq u_\alpha (z_\alpha)\\
&\geq u_\alpha (x_\alpha) + p\cdot (z_\alpha- x_\alpha)\\
&= (1-\alpha)u_0(x_0) +\alpha u_1(x_1) + (1-\alpha) p\cdot (z_0-x_0). 
\end{split}
\end{equation*}
Since $\alpha\in (0, 1)$, it follows that 
\[u_0(z_0)\geq u_0(x_0) + p\cdot (z_0-x_0).\]
This shows that $p\in \p u_0(x_0)$. By a property of the Legendre transform, we then have $x_0\in \p u_0^* (p)$. Since $u_0$ is strictly convex, $u_0^*\in C^1(\p u_{0}(\Omega_0))$, so $x_0= Du_0^*(p)$. Part (i) is proved.

\medskip
We prove part (ii). Assume $x_1\in\p\Omega_1$. We prove that $p\not \in \p u_{1}(\Omega_1)$. If otherwise, $p\in \p u_1(y_1)$ where $y_1\in\Omega_1$. Let $w_\alpha= (1-\alpha)x_0 +\alpha y_1$. 
Then
\begin{equation*}
\begin{split}
u_\alpha(x_\alpha) &= (1-\alpha)u_0(x_0) +\alpha u_1(x_1) \\
&\geq (1-\alpha)u_0(x_0) + \alpha [u_1(y_1) + p\cdot (x_1-y_1)]\\
&=  (1-\alpha)u_0(x_0)  +\alpha u_1 (y_1) + p\cdot (x_\alpha - w_\alpha)\\
&\geq u_\alpha(w_\alpha) + p\cdot (x_\alpha - w_\alpha)  \quad (\text{by the definition of } u_\alpha)\\
&\geq u_\alpha(x_\alpha) \quad (\text{since } p\in \p u_{\alpha}(x_\alpha)).
\end{split}
\end{equation*}
Thus, we must have all equalities. In particular, we must have
\[u_1(x_1) = u_1(y_1) + p\cdot (x_1-y_1).\]
This is impossible due to the strict convexity of $u_1$ and $y_1\in\Omega_1$ and $x_1\in\p\Omega_1$.

The proposition is proved.
\end{proof}

Although the union of $\p u_{0}(\Omega_0)$ and $\p u_{1}(\Omega_1)$ might strictly contain $\p u_{\alpha}(\Omega_\alpha)$, we show that their intersection is contained in $\p u_{\alpha}(\Omega_\alpha)$.
\begin{prop}
\label{obs2} 
Assume \eqref{Sta}. Let $u_\alpha$ be as in \eqref{infc}. Let $p\in \p u_0(x_0) \cap  \p u_1(x_1)$ where $x_0\in\Omega_0$ and $x_1\in\Omega_1$.  Let
$x_\alpha= (1-\alpha)x_0 +\alpha x_1\in\Omega_\alpha$.
Then
\[p\in  \p u_{\alpha}(x_\alpha)\quad\text{and}\quad u_\alpha(x_\alpha)=(1-\alpha)u_0(x_0) +\alpha u_1(x_1).\]
Consequently,
\[ \p u_{0}(\Omega_0)\cap  \p u_{1}(\Omega_1)\subset \p u_{\alpha}(\Omega_\alpha).\]
\end{prop}
\begin{proof} 
Consider $z_\alpha\in\Omega_\alpha$. By Lemma \ref{Sal8}, 
there exists a unique couple of points $(z_0, z_1)\in \overline{\Omega_0} \times \overline{\Omega_1}$ such that
\[z_\alpha= (1-\alpha)z_0 +\alpha z_1\quad \text{and}\quad u_\alpha(z_\alpha)=(1-\alpha)u_0(z_0) +\alpha u_1(z_1).\]
We have from the definition of subdifferential and $p\in \p u_0(x_0) \cap  \p u_1(x_1)$ that
\begin{equation*}
\begin{split}
u_\alpha(z_\alpha)&=(1-\alpha)u_0(z_0) +\alpha u_1(z_1) \\
&\geq (1-\alpha)[u_0(x_0) + p\cdot (z_0-x_0)] + \alpha [u_1(x_1) + p\cdot (z_1-x_1)]\\
&= (1-\alpha)u_0(x_0) +\alpha u_1(x_1) + p\cdot (z_\alpha-x_\alpha)\\
&\geq u_\alpha (x_\alpha)  + p\cdot (z_\alpha-x_\alpha) \quad (\text{by the definition of } u_\alpha).
\end{split}
\end{equation*}
It follows that $p\in  \p u_{\alpha}(x_\alpha)$. By choosing $z_\alpha=x_\alpha$, we find that all the above inequalities must be equalities. In particular, $u_\alpha(x_\alpha)=(1-\alpha)u_0(x_0) +\alpha u_1(x_1)$. The proposition is proved.
\end{proof}
Lemma \ref{Sal8} and Propositions \ref{obs1} and \ref{obs2} motivate the following decomposition.
\begin{defn} Assume \eqref{Sta}. Let $u_\alpha$ be as in \eqref{infc}. We decompose 
\begin{equation}
\label{3sd}
\p u_{\alpha}(\Omega_\alpha) = S_1\cup S_2\cup S_3,\end{equation}
where
\begin{equation*}
\begin{split}
S_1&=\{p\in  \p u_{\alpha}(x_\alpha): x_\alpha=(1-\alpha)x_0 +\alpha x_1,\, \\ & \quad\quad\quad\quad\quad\quad u_\alpha(x_\alpha)=(1-\alpha)u_0(x_0) +\alpha u_1(x_1), x_0\in\Omega_0, x_1\in\Omega_1\}, \\
S_2&=\{p\in  \p u_{\alpha}(x_\alpha): x_\alpha=(1-\alpha)x_0 +\alpha x_1,\,  \\ & \quad\quad\quad\quad\quad\quad  u_\alpha(x_\alpha)=(1-\alpha)u_0(x_0) +\alpha u_1(x_1), x_0\in\Omega_0, x_1 \in\p\Omega_1\}, \\
S_3&=\{p\in  \p u_{\alpha}(x_\alpha): x_\alpha=(1-\alpha)x_0 +\alpha x_1,\,  \\ & \quad\quad\quad\quad\quad\quad  u_\alpha(x_\alpha)=(1-\alpha)u_0(x_0) +\alpha u_1(x_1), x_0\in\p\Omega_0, x_1\in\Omega_1\}.
\end{split}
\end{equation*}
\end{defn}
The effect of the decomposition \eqref{3sd} on the map $u_\alpha \circ D u^*_\alpha$ is as follows.
\begin{rem}
\label{obs3}
Assume \eqref{Sta}. Let $u_\alpha$ be as in \eqref{infc}. Consider the decomposition \eqref{3sd}.
From Propositions \ref{obs1} and \ref{obs2}, we have 
\[S_1= \p u_{0}(\Omega_0)\cap  \p u_{1}(\Omega_1),\quad  S_2\subset \p u_{0}(\Omega_0)\setminus  \p u_{1}(\Omega_1),\quad S_3\subset \p u_{1}(\Omega_1)\setminus  \p u_0(\Omega_0).\]
If $p\in S_1$, then Proposition \ref{obs1}
gives
\[u_\alpha (D u^*_\alpha(p)) = (1-\alpha) u_0 (D u^*_0(p)) + \alpha u_1 (D u^*_1(p)).\]
If $p\in S_2$, then Proposition \ref{obs1}
gives
\[u_\alpha (D u^*_\alpha(p)) = (1-\alpha) u_0 (D u^*_0(p)).\]
If $p\in S_3$, then Proposition \ref{obs1}
gives
\[u_\alpha (D u^*_\alpha(p)) = \alpha u_1 (D u^*_1(p)).\]
\end{rem}

\medskip
We have the following sublinearity of the Monge--Amp\`ere energy $\int_\Omega |u| d\mu_u$ under the infimum convolution of convex functions.
\begin{prop} 
\label{obs4}
Assume \eqref{Sta}. Let $u_\alpha$ be as in \eqref{infc}.
Then
\[\int_{\Omega_\alpha} |u_\alpha| d\mu_{u_\alpha} \leq (1-\alpha) \int_{\Omega_0} |u_0| d\mu_{u_0} + \alpha \int_{\Omega_1} |u_1| d\mu_{u_1}.\]
\end{prop}
\begin{proof} Since $u_0, u_1, u_\alpha\leq 0$, it suffices to show that 
\[\int_{\Omega_\alpha} u_\alpha d\mu_{u_\alpha} \geq (1-\alpha) \int_{\Omega_0} u_0 d\mu_{u_0} + \alpha \int_{\Omega_1} u_1 d\mu_{u_1}.\]
By Lemma \ref{change_lem}, this is equivalent to
\begin{equation}
\label{uadomi}
\int_{\p u_\alpha(\Omega_\alpha)} u_\alpha (D u^*_\alpha(p)) \, dp \geq \int_{\p u_0(\Omega_0)} u_0 (D u^*_0(p)) \, dp +  \int_{\p u_1(\Omega_1)} u_1 (D u^*_1(p)) \,dp.\end{equation}
Consider the decomposition \eqref{3sd}.

From Remark \ref{obs3},
we deduce that
\begin{equation}
\label{S1ineq}
\begin{split}
\int_{S_1} u_\alpha (D u^*_\alpha(p))\, dp 
&= (1-\alpha) \int_{\p u_{0}(\Omega_0)\cap  \p u_{1}(\Omega_1)} u_0 (D u^*_0(p))\,  dp\\ &\quad + \alpha \int_{\p u_{0}(\Omega_0)\cap  \p u_{1}(\Omega_1)}u_1 (D u^*_1(p))\, dp.
\end{split}
\end{equation}
From Remark \ref{obs3} and $u_0, u_1\leq 0$, we deduce that
\begin{equation}
\begin{split}
\int_{S_2} u_\alpha (D u^*_\alpha(p))\, dp &= (1-\alpha) \int_{S_2} u_0 (D u^*_0(p))\,  dp\\
&\geq (1-\alpha) \int_{\p u_{0}(\Omega_0)\setminus  \p u_{1}(\Omega_1)} u_0 (D u^*_0(p))\,  dp\\
&\geq (1-\alpha) \int_{\p u_{0}(\Omega_0)\setminus  \p u_{1}(\Omega_1)} u_0 (D u^*_0(p))\,  dp \\
&\quad \quad + \alpha \int_{\p u_{0}(\Omega_0)\setminus  \p u_{1}(\Omega_1)} u_1 (D u^*_1(p))\,  dp.
\end{split}
\end{equation}
Similarly,
\begin{equation}
\label{S3ineq}
\begin{split}
\int_{S_3} u_\alpha (D u^*_\alpha(p))\, dp 
&\geq (1-\alpha) \int_{\p u_{1}(\Omega_1)\setminus  \p u_0(\Omega_0)} u_0 (D u^*_0(p))\,  dp \\
&\quad \quad + \alpha \int_{\p u_{1}(\Omega_1)\setminus  \p u_0(\Omega_0)} u_1 (D u^*_1(p))\,  dp.
\end{split}
\end{equation}
Combining \eqref{S1ineq}--\eqref{S3ineq}, we obtain \eqref{uadomi}. The proposition is proved.
\end{proof}
\section{Proof of the Brunn--Minkowski inequality and equality case}
\label{sec_pf}
In this section, we prove Theorems \ref{BMcor} and \ref{eqthm}.

We recall \cite[Lemma 13]{Sal} and the discussion in \cite[p. 83]{Sal}.
\begin{lem} 
\label{Sal13}
Let $\Omega_0$ and $\Omega_1$ be bounded convex domains in $\R^n$ and $\alpha\in (0, 1)$. 
Let $\Omega_\alpha =(1-\alpha)\Omega_0 +\alpha \Omega_1$.
For each $i=0, 1$, let $u_i\in C(\overline{\Omega_i})$ be a strictly convex function 
in $\Omega_i$ that vanishes on $\p\Omega_i$ and satisfies $\int_{\Omega_i}|u_i|^{n+1}\, dx=1$.  Let $u_\alpha$ be the infimal convolution of $u_0$ and $u_1$. Then
\begin{equation}
\label{PLa}
\int_{\Omega_\alpha}|u_\alpha|^{n+1}\, dx\geq 1.\end{equation}
Equality happens if and only $u_0(x) = m^{n/(n+1)} u_1(mx + b)$ and $\Omega_1= m\Omega_0 + b$ for some $m>0$ and some $b\in\R^n$.
\end{lem}
For the reader's convenience, we include the simple argument.
\begin{proof}[Proof of Lemma \ref{Sal13}] From \eqref{infc}, we have for each $x\in\Omega_\alpha$
\[|u_\alpha(x)| =\max\big\{ (1-\alpha)|u_0(x_0)| +\alpha |u_1(x_1)|:  x_0\in  \overline{\Omega_0}, x_1\in  \overline{\Omega_1}, x= (1-\alpha)x_0 +\alpha x_1\big\}. \]
By the arithmetic-geometric inequality, we find
 \[|u_\alpha(x)| \geq \max\big\{ |u_0(x_0)|^{1-\alpha} |u_1(x_1)|^\alpha:  x_0\in  \overline{\Omega_0}, x_1\in  \overline{\Omega_1}, x= (1-\alpha)x_0 +\alpha x_1\big\}. \]
 Let $v_i(x)=|u_i(x)|^{n+1}$ for $x\in \Omega_i$ and $v_i(x) =0$ for $x\in\R^n\setminus \Omega_i$ where $i=0, 1,\alpha$. Then
  \[v_\alpha(x) \geq v_0(x_0)^{1-\alpha} v_1(x_1)^\alpha\quad \quad\text{whenever }\quad x= (1-\alpha)x_0 +\alpha x_1. \]
  By the Pr\'ekopa--Leindler inequality (see \cite[Theorem 6.4]{V}), we have
  \begin{equation}
\label{PLa1}\int_{\R^n} v_\alpha(x)\, dx \geq \Big(\int_{\R^n} v_0(x)\, dx\Big)^{1-\alpha} \Big(\int_{\R^n} v_1(x)\, dx\Big)^{\alpha}.\end{equation}
  This implies \eqref{PLa}.
  
  \medskip
  Equality in \eqref{PLa} happens if and only equality in \eqref{PLa1} happens. By the equality case of the Pr\'ekopa--Leindler inequality in Dubuc \cite[Theorem 12]{Dub}, this happens only if
  there exist $m>0$ and $b\in\R^n$ such that
  \[v_0(x) = \frac{\int_{\R^n} v_0(x)\, dx}{\int_{\R^n} v_1(x)\, dx} m^n v_1(mx + b)\quad\text{for almost everywhere } x\in\R^n.\]
  This implies $\Omega_1= m\Omega_0 + b$ and
    \[u_0(x) =  m^{n/(n+1)} u_1(mx + b)\quad\text{for all } x\in\Omega_0.\]
The lemma is proved.
\end{proof}
\begin{proof}[Proof of Theorems \ref{BMcor} and \ref{eqthm}] We first prove the ``if'' part in Theorem \ref{eqthm}.
Suppose that $\Omega_1$ is homothetic to $\Omega_0$, that is, $\Omega_1= \beta \Omega_0 + b$ for some $\beta>0$ and $b\in\R^n$. Then 
\[\Omega_\alpha:=(1-\alpha)\Omega_0 + \alpha\Omega_1 =(1-\alpha + \alpha\beta)\Omega_0  + \alpha b. \]
We compute, using the homogeneity of degree $(-2n)$ of $\lambda[\cdot]$,
\[\lambda[\Omega_1] = \beta^{-2n} \lambda[\Omega_0],\quad \lambda[\Omega_\alpha]= (1-\alpha + \alpha\beta)^{-2n}\lambda[\Omega_0]. \]
Therefore,
\[\lambda\big[(1-\alpha)\Omega_0 + \alpha\Omega_1\big]^{-\frac{1}{2n}}=  (1-\alpha + \alpha\beta)\lambda[\Omega_0]^{-\frac{1}{2n}}= (1-\alpha) \lambda [\Omega_0]^{-\frac{1}{2n}} + \alpha \lambda[\Omega_1]^{-\frac{1}{2n}}.\]

\medskip
We now prove Theorem \ref{BMcor} and the ``only if'' part in Theorem \ref{eqthm}.
\medskip

\noindent
{\bf Step 1.}
Let $\Omega_0$ and $\Omega_1$ be bounded convex domains in $\R^n$ and $\alpha\in (0, 1)$. We claim that 
\begin{equation}
\label{BMequi1}
\lambda[\Omega_\alpha]
\leq (1-\alpha) \lambda[\Omega_0] + \alpha\lambda[\Omega_1],
\end{equation}
and equality occurs only if $\Omega_1$ is homothetic to $\Omega_0$.

Indeed, let $u_i\in C^{\infty}(\Omega_i)\cap C^{0, 1}(\overline{\Omega_i})$ be the Monge--Amp\`ere eigenfunction of $\Omega_i$ with 
\[u_i=0\quad\text{on}\quad \p\Omega_i \quad\text{and}\quad \int_{\Omega_i}|u_i|^{n+1}\, dx=1.\]
Then
\[\lambda[\Omega_i]=\int_{\Omega_i} |u_i| \, d\mu_{u_i}.\]
Let $u_\alpha$ be the infimal convolution of $u_0$ and $u_1$. Then, by Lemma \ref{Sal8}, $u_\alpha \in  C(\overline{\Omega_\alpha})$ is strictly convex with the following property:
\begin{equation*}
u_\alpha = 0\quad \text{on }\p\Omega_\alpha,\quad \text{and } \quad u_\alpha < 0\quad \text{in }\Omega_\alpha.\end{equation*}
Due to $u_i\in C^{0, 1}(\overline{\Omega_i})$, we have in fact $u_\alpha \in  C^{0, 1}(\overline{\Omega_\alpha})$, by Proposition \ref{obs1} (i).

From \eqref{lam_def}, Lemma \ref{Sal13} and Proposition \ref{obs4}, we have
\begin{equation*}
\begin{split}
\lambda[\Omega_\alpha] \leq \frac{\int_{\Omega_\alpha} |u_\alpha|\, d\mu_{u_\alpha} }{\int_{\Omega_\alpha}|u_\alpha|^{n+1}\, dx}&\leq
\int_{\Omega_\alpha} |u_\alpha|\, d\mu_{u_\alpha}  \\&\leq (1-\alpha) \int_{\Omega_0} |u_0| d\mu_{u_0} + \alpha \int_{\Omega_1} |u_1| d\mu_{u_1}\\
&=(1-\alpha) \lambda[\Omega_0] + \alpha\lambda[\Omega_1],
\end{split}
\end{equation*}
proving the claimed inequality \eqref{BMequi1}.
Equality occurs only if equality occurs in Lemma \ref{Sal13}, that is $\Omega_1$ is homothetic to $\Omega_0$.

\medskip
\noindent
{\bf Step 2.} We prove the Brunn--Minkowski inequality and determine the equality case. Let $\Omega_0$ and $\Omega_1$ be bounded convex domains in $\R^n$ and $\alpha\in (0, 1)$.
Let us consider
\[\Omega'_0 = \lambda [\Omega_0]^{\frac{1}{2n}}\Omega_0,\quad \Omega'_1 = \lambda [\Omega_1]^{\frac{1}{2n}}\Omega_1,\quad \alpha'= \frac{\alpha \lambda[\Omega_1]^{-\frac{1}{2n}}}{(1-\alpha) \lambda [\Omega_0]^{-\frac{1}{2n}} + \alpha \lambda[\Omega_1]^{-\frac{1}{2n}}}.\]
Then
\[\lambda[\Omega'_0]=\lambda[\Omega'_1]=1,\]
and
\[(1-\alpha')\Omega'_0 +\alpha' \Omega'_1 = \frac{1}{(1-\alpha) \lambda [\Omega_0]^{-\frac{1}{2n}} + \alpha \lambda[\Omega_1]^{-\frac{1}{2n}}} \Big( (1-\alpha)\Omega_0 +\alpha \Omega_1 \Big).\]
Applying Step 1 to $\Omega'_0,\Omega'_1$ with $\alpha'\in (0, 1)$, we have
\[\lambda[(1-\alpha')\Omega'_0 +\alpha' \Omega'_1]\leq (1-\alpha') \lambda[\Omega'_0] +\alpha' \lambda[\Omega'_1]=1.\]
This means that
\[\Big((1-\alpha) \lambda [\Omega_0]^{-\frac{1}{2n}} + \alpha \lambda[\Omega_1]^{-\frac{1}{2n}}\Big)^{2n} \lambda[(1-\alpha)\Omega_0 +\alpha \Omega_1]\leq 1,\]
which is exactly
\[\lambda\big[(1-\alpha)\Omega_0 + \alpha\Omega_1\big]^{-\frac{1}{2n}}\geq (1-\alpha) \lambda [\Omega_0]^{-\frac{1}{2n}} + \alpha \lambda[\Omega_1]^{-\frac{1}{2n}}.\]
This proves Theorem \ref{BMcor}. 

From Step 1, we see that
equality case happens only if $\Omega'_1$ is homothetic to $\Omega'_0$, which is equivalent to $\Omega_1$ being homothetic to $\Omega_0$. This completes the proof of Theorem \ref{eqthm}.
\end{proof}

\end{document}